\documentclass[a4paper,reqno]{amsart}
\usepackage{amsmath}
\usepackage{bm}
\usepackage{amsfonts}
\usepackage{amsthm}
\usepackage{amssymb}
\usepackage{mathrsfs}
\usepackage{booktabs}
\usepackage{lipsum}
\usepackage{bussproofs}
\usepackage{hyperref}
\usepackage{tikz}
\usetikzlibrary{arrows,automata,positioning,decorations.markings,patterns}
\usepackage{python}

\usepackage{amsmath}
\usepackage{amsfonts}
\usepackage{amsthm}
\usepackage{amssymb}
\usepackage{mathrsfs}

\def\parentheses#1{{\left( {#1} \right)}}
\def\of{\parentheses}

\def\Fraction#1/#2{\frac{#1}{#2}}
\def\fraction#1/#2{\tfrac{#1}{#2}}
\def\solidus#1/#2{%
    {{#1} \!\left/\!\vphantom{{#1}{#2}}\!\right. {#2}}%
}

\def\blackboard{\mathbb}

\def\infinity{\infty}

\def\naturals{\blackboard N}

\def\reals{\blackboard R}

\def\set#1{{\left\lbrace {#1} \right\rbrace}}
\def\setOf#1:#2{%
    \set{{#1} \,\left\vert\vphantom{{#1}{#2}}\right. {#2}}%
}

\def\emptySet{\varnothing}

\def\sequence#1{{\left\langle {#1} \right\rangle}}
\def\sequenceOf#1:#2{%
    \sequence{{#1} \,\left\vert\vphantom{{#1}{#2}}\right. {#2}}%
}

\def\familyOf#1:#2{{\parentheses{#1}_{#2}}}

\def\union{\cup}

\def\intersect{\cap}

\def\setMinus{\setminus}

\def\functionSpace#1->#2{{{#2}^{#1}}}

\def\function#1:#2->#3{{{#1} \colon {#2} \to {#3}}}
\def\mapsTo{\mapsto}
\def\map#1->#2{\parentheses{{#1} \mapsTo {#2}}}

\def\Restriction#1|#2{{\left.\! {#1} \right\vert_{#2}}} 

\def\presentation#1{{\left\langle {#1} \right\rangle}}
\def\presentationOf#1:#2{%
    \presentation{{#1} \,\left\vert\vphantom{{#1}{#2}}\right. {#2}}%
}

\def\interval#1#2,#3#4{{\left#1 {#2}, {#3} \right#4}}

\def\Graph#1{{\Gamma\of{#1}}}

\def\p{\parentheses}

\newtheorem{presentment}{}[section]

\theoremstyle{plain}
\newtheorem{theorem}[presentment]{Theorem}
\newtheorem{proposition}[presentment]{Proposition}
\newtheorem{lemma}[presentment]{Lemma}

\theoremstyle{definition}

\newtheorem{example}[presentment]{Example}

\newtheorem{question}[presentment]{Question}

\theoremstyle{remark}

\newtheorem*{note}{Note}

    \def\fraction#1/#2{\tfrac{#1}{#2}}
    \def\Fraction#1/#2{\frac{#1}{#2}}
\newcommand*{\BigFraction}{}
    \def\BigFraction#1/#2{\dfrac{\displaystyle{#1}}{\displaystyle{#2}}}

    \def\setOf#1:#2{\set{{#1} \,\left\vert\vphantom{{#1}{#2}}\right. {#2}}}

    \def\sequenceOf#1:#2{\sequence{{#1} \,\left\vert\vphantom{{#1}{#2}}\right. {#2}}}

    \def\function#1:#2->#3{{{#1} \colon {#2} \to {#3}}}

\tikzset{
    o/.style={
        shorten >=#1,
        decoration={
            markings,
            mark={
                at position 1
                with {
                    \draw circle [radius=#1];
                }
            }
        },
        postaction=decorate
    },
    o/.default=2pt
}

\usepackage{pst-node}
\usepackage{tikz-cd} 
\tikzset{node distance=2cm, auto}

\allowdisplaybreaks[0]
\title{The Set-Self-Tietze Property} 
\author{Andrew Wood}


\email {} \email {andrew.wood@stcatz.ox.ac.uk}

\address
{Mathematical Institute, University of Oxford, Andrew Wiles Building
Radcliffe Observatory Quarter 
Woodstock Road
Oxford
OX2 6GG }

\begin{document}

\maketitle

\begin{abstract}
We introduce the set-self-Tietze property, an analogue of the self-Tietze property for upper semi-continuous set-valued functions.  A topological space $X$ is \emph{self-Tietze}, if for every closed $A \subseteq X$ and continuous function $f : A \to X$, there is a continuous extension $F : X \to X$ of $f$.  A topological space $X$ is \emph{set-self-Tietze}, if for every closed $A \subseteq X$ and upper semi-continuous set-valued function $f : A \to 2^X$, there exists an upper semi-continuous set-valued function $F : X \to 2^X$ such that $\left. F \right|_A = f$.  We show every compact metric space is set-self-Tietze, and that the torus is not self-Tietze.   
\end{abstract}



\section{Introduction}
\renewcommand{\thepresentment}{\Alph{presentment}}

The Tietze Extension Theorem, which states that a space $X$ is normal if, and only if, every real-valued continuous function defined on a closed subspace of $X$ can be extended to the entire space $X$, is a well-known extension result in topology. It is natural to then ask what happens if one changes the codomain of these continuous functions from $\reals$ to our given space $X$.  This leads us to the notion of a self-Tietze space, introduced in~\cite{davies2011a}.  More precisely, a space $X$ is {\color{blue} \emph{self-Tietze}}, if every continuous function defined on a closed subspace of $X$ to $X$ can be extended to a continuous self-map on $X$. Immediately from the Tietze Extension Theorem, one observes both $[0, 1]$ and $\reals$ are self-Tietze.   It is shown in~\cite{davies2011a} that the Tychonoff Plank is self-Tietze, yet the irrational double-arrow space is not.  We will show the torus is not self-Tietze.  

In this note, we introduce a new topological property analogous to self-Tietze but for upper semi-continuous (usc) set-valued functions $f : X \to 2^Y$ (where $2^Y$ is the collection of non-empty closed subsets of $Y$).  We say a topological space $X$ is {\color{blue} \emph{set-self-Tietze}}, if for every closed $A \subseteq X$ and usc $f : A \to 2^X$ there exists usc extension $F : X \to 2^X$ of $f$. The main result we prove is shown below.  

\begin{theorem}\label{theorem:compact-metric-is-set-self-tietze}
Every compact metric space $X$ is set-self-Tietze. 
\end{theorem}

\begin{note}
As an immediate consequence, every CR-dynamical system $\p{X, G}$ on a compact metric space $X$ extends to a set-valued dynamical system.  See~\cite{minimality_CR, transitivity_CR, nagar_transitivity_CR, wood2025shadowingcrdynamicalsystems, banic2025specificationmahaviersystemsclosed, greenwood2025transitivitycrdynamicalsystems} for literature on CR-dynamical systems. 
\end{note}

Indeed, as we show the torus is not self-Tietze, we find the analoguous statement of Theorem~\ref{theorem:compact-metric-is-set-self-tietze} for self-Tietze does not hold.

\renewcommand{\thepresentment}{\thesection.\arabic{presentment}}

\section{Proof of Theorem~\ref{theorem:compact-metric-is-set-self-tietze}}

In this section, we prove our main result, Theorem~\ref{theorem:compact-metric-is-set-self-tietze}. 

\begin{proof}
Suppose $A$ is closed in $X$ and $f : A \to 2^X$ is usc. Define 
\[
G = \setOf{\p{x, y} \in X \times A}:{d\of{x, y} = \text{dist}\of{x, A}}.
\]
We claim $G$ is a closed in $X \times A$ with $p_1\of{G} = X$, where $p_1$ is projection onto first coordinate. To see why $p_1\of{G} = X$, we observe for each $x \in X$ that dist$\of{x, -} : A \to \reals$ is continuous on the compact space $A$, thus attaining its minimum.  Next, suppose $\p{x_n, y_n} \in G$ for each $n \in \naturals$ and $\lim_{n \to \infinity} \p{x_n, y_n} = \p{x, y}$. We show $\p{x, y} \in G$. Indeed, $\p{x, y} \in X \times A$, because $X \times A$ is closed and $\sequenceOf{\p{x_n, y_n}}:{n \in \naturals}$ is a sequence in $X \times A$. Recall metrics are continuous, so using sequential characterisation of continuity we obtain 
\[
\lim_{n\to\infinity} d\of{x_n, y_n} = d\of{x, y}. 
\]
On the other hand, dist$\p{-, A} : X \to \reals$ is continuous, so using sequential characterisation of continuity we obtain 
\[
\lim_{n\to\infinity} \text{dist}\of{x_n, A} = \text{dist}\of{x, A}.
\]
For each $n \in \naturals$, $\p{x_n, y_n} \in G$, which implies $d\of{x_n, y_n} = \text{dist}\of{x_n, A}$. Hence, $d\of{x, y} = \text{dist}\of{x, A}$, which means $\p{x, y} \in G$. It follows $G$ is closed in $X \times A$ with $p_1\of{G} = X$, as claimed. Because $X$ and $A$ are compact metric spaces, there exists usc $g : X \to 2^A$ with $\Graph{g} = G$, where $\Graph{g}$ denotes the graph of $g$. 

Now, let $F = f \circ g$. Observe $F : X \to 2^X$ is usc, since it is the composition of usc functions $f$ and $g$. Moreover, for each $x \in A$, $F\of{x} = f\of{g\of{x}} = f\of{x}$ because 
\[
\p{x, y} \in G \Longleftrightarrow d\of{x, y} = \text{dist}\of{x, A} = 0 \Longleftrightarrow x = y.
\]
Thus, $X$ is set-self-Tietze. 
\end{proof}

We pose the following natural question as a next step.  

\begin{question}
Is every compact Hausdorff space set-self-Tietze? 
\end{question}

For example, the Tychonoff Plank is compact Hausdorff and non-metrizable. It is known the Tychonoff Plank is self-Tietze.

\begin{question}
Is the Tychonoff Plank set-self-Tietze?
\end{question}

\section{The torus is not self-Tietze}
In this section, we prove that the torus is not self-Tietze.  It turns out, this can be readily seen as a consequence of two results. The first appears in~\cite{davies2011a}, and the second is well-known.  

\begin{lemma}
$\p{\text{\emph{The Self-Tietze Projection Characterisation}}}$. A topological space $X = \prod_{i \in I} X_i$ is self-Tietze if, and only if, for every $i \in I$, every closed $C \subseteq X$, and every continuous $f : C \to X_i$, there is a continuous extension $F : X \to X_i$ of $f$. 
\end{lemma}

\begin{lemma}
The Lebesgue covering dimension of a normal space $X$ is $\leq n$ if, and only if, for every closed $C \subseteq X$ and continuous $f : C \to S^n$, there is a continuous extension $F : X \to S^n$ of $f$. 
\end{lemma}

\begin{proposition}
The torus $X = S^1 \times S^1$ is not self-Tietze, but it is set-self-Tietze. 
\end{proposition}
\begin{proof}
Observe $X$ is a compact metric space, which implies it is set-self-Tietze by Theorem~\ref{theorem:compact-metric-is-set-self-tietze}. To see why $X$ is not self-Tietze, we firstly note the covering dimension of $X$ is $2$. It follows by the above lemma that there is closed $C \subseteq X$ and continuous $f : C \to S^1$ such that $f$ has no continuous extension $F : X \to S^1$. By the Self-Tietze Projection Characterisation, this implies $X$ is not self-Tietze. 
\end{proof}

\section{Comparing to self-Tietze}
In this section, we compare the self-Tietze property to the set-self-Tietze property.  We consider results that hold for self-Tietze spaces, proved in~\cite{davies2011a}, and determine whether they hold for set-self-Tietze as well. Throughout this section, we assume all spaces are at least $T_1$.  

Every Hausdorff self-Tietze space is normal. We find the analogous statement for set-self-Tietze holds too.

\begin{proposition}
Every Hausdorff set-self-Tietze space $X$ is normal. 
\end{proposition}
\begin{proof}
Let $A$ and $B$ be non-empty disjoint closed subsets of $X$, and fix $\p{a, b} \in A \times B$. Define $f : A \union B \to 2^X$ such that $f\of{A} = \set{a}$ and $f\of{B} = \set{b}$. Note $f$ is well-defined because $A$ and $B$ are disjoint, and $X$ is $T_1$. Since $A$ and $B$ are both closed in $A \union B$, it follows $f$ is usc. By the set-self-Tietze property, there exists usc extension $F : X \to 2^X$ of $f$. By the Hausdorff property, there exists disjoint open nhoods $U$ and $V$ of $a$ and $b$, respectively, in $X$. Note $O_A = \setOf{x \in X}:{F\of{x} \subseteq U}$ and $O_B = \setOf{x \in X}:{F\of{x} \subseteq V}$ are open in $X$. Observe $O_A$ and $O_B$ are disjoint because $U$ and $V$ are disjoint. Moreover, $A \subseteq O_A$ and $B \subseteq O_B$. Thus, $X$ is normal. 
\end{proof}

Every disconnected, self-Tietze space is ultranormal.  We find this is not the case for set-self-Tietze.

\begin{example}
Consider $X = [0, 1] \union \set{2}$, which is set-self-Tietze by Theorem~\ref{theorem:compact-metric-is-set-self-tietze}. Moreover, $X$ is disconnected. However, it is not ultranormal, since $A = [0, \Fraction 1 / 4]$ and $B = [\Fraction 3/4, 1]$ are disjoint closed subsets of $X$, but the only non-trivial clopen partition is $\set{[0, 1], \set{2}}$.
\end{example}

Although we may not deduce ultranormality by disconnectedness of a set-self-Tietze space, we may deduce normality.

\begin{proposition}
If a set-self-Tietze $X$ is disconnected, then it is normal. 
\end{proposition}
\begin{proof}
Let $A$ and $B$ be disjoint closed subsets of $X$, and suppose $E$ is a non-trivial clopen subset of $X$. Fix $a \in E$ and $b \in X \setMinus E$. Define $f : A \union B \to X$ such that $f\of{A} = \set{a}$ and $f\of{B} = \set{b}$. Observe $f$ is well-defined because $A$ and $B$ are disjoint, and $X$ is $T_1$. Since $A$ and $B$ are both closed in $A \union B$, it follows $f$ is usc. As $X$ is set-self-Tietze and $A \union B$ is closed in $X$, there exists usc extension $F : X \to 2^X$ of $f$. It follows $F^{-1}\of{E}$ is closed in $X$ and $O = \setOf{x \in X}:{F\of{x} \subseteq E}$ is open in $X$. Notice $A \subseteq O \subseteq F^{-1}\of{E}$.  That is to say, $F^{-1}\of{E}$ is a closed nhood of $A$, disjoint from $B$. Thus, $X$ is normal. 
\end{proof}

A topological space $X$ is {\color{blue} \emph{retractifiable}}, if every closed subset $A$ of $X$ is a retract, i.e., there is a continuous map $f : X \to A$ such that $\left. f \right|_A = \text{id}_A$. For example, ordinals are retractifiable, and so are separable, metric, zero-dimensional spaces. Retractifiable spaces are self-Tietze.  We end by showing this is true of set-self-Tietze spaces as well.

\begin{proposition}
Every retractifiable space $X$ is set-self-Tietze.
\end{proposition}
\begin{proof}
Let $A$ be a closed subset of $X$ and $f : A \to 2^X$ be usc. There exists continuous $r : X \to A$ such that $\left. r\right|_A = \text{id}_A$. Define $F : X \to 2^X$ by $F\of{x} = f\of{r\of{x}}$ for each $x \in X$. Note $F\of{x} = f\of{x}$ for each $x \in A$. Let $C$ be closed in $X$. Observe
\begin{align*}
&&  F^{-1}\of{C} &= \setOf{x \in X}:{f\of{r\of{x}} \intersect C \neq \emptySet}   && \\
&&  &= \setOf{x \in X}:{r\of{x} \in f^{-1}\of{C}} && \\
&& &= r^{-1}\of{f^{-1}\of{C}}.  &&
\end{align*}
It follows $F^{-1}\of{C}$ is closed in $X$ because $f$ is usc, $r$ is continuous, and $C$ is closed in $X$. Thus, $X$ is set-self-Tietze. 
\end{proof}

\bibliographystyle{plain}
\bibliography{references}

\end{document}